\documentclass[3p,10pt,a4paper,twoside,fleqn,sort&compress]{filomat}
\usepackage{amssymb,amsmath,latexsym}
\usepackage[varg]{pxfonts}

\newtheorem{theorem}{Theorem}[section]

\newtheorem{definition}[theorem]{Definition}

\newtheorem{proposition}[theorem]{Proposition}

\def\R{\mathbb R}

\def\K{\mathbb K}

\newcommand{\FilVol}{xx}
\newcommand{\FilYear}{yyyy}
\newcommand{\FilPages}{zzz--zzz}
\newcommand{\FilDOI}{(will be added later)}
\newcommand{\FilReceived}{Received: dd Month yyyy; Revised: dd Month yyyy; Accepted: dd Month yyyy}

\begin{document}

\title{Corrigendum to ``On the Mazur-Ulam theorem in non-Archimedean fuzzy anti-2-normed spaces"}

\author[affil1]{Javier~Cabello~Sánchez}
\ead{coco@unex.es}
\author[affil2]{José~Navarro~Garmendia}
\ead{navarrogarmendia@unex.es}
\address[affil1]{Departamento de Matem\'{a}ticas, Universidad de Extremadura, Avenida de Elvas s/n, 06006; Badajoz. Spain}
\address[affil2]{Departamento de Matem\'{a}ticas, Universidad de Extremadura, Avenida de Elvas s/n, 06006; Badajoz. Spain}
\newcommand{\AuthorNames}{F. Author, S. Author}

\newcommand{\FilMSC}{Primary 46S10; Secondary 26E30}
\newcommand{\FilKeywords}{fuzzy normed spaces; Mazur-Ulam Theorem; 
non-Archimedean normed spaces; strict convexity}
\newcommand{\FilCommunicated}{(name of the Editor, mandatory)}
\newcommand{\FilSupport}{Research supported in part by DGICYT project PID2019-103961GB-C21 (Spain), ERDF funds and Junta de Extremadura programs IB-16056 and IB-18087}

\begin{abstract}
In this note we correct a paper by D. Kang (``On the Mazur-Ulam theorem 
in non-Archimedean fuzzy anti-2-normed spaces", Filomat, 2017).

The research in that paper applies to what the author calls strictly convex spaces. Nevertheless, we prove that this notion is void: there is no single space that satisfies the definition.
\end{abstract}

\maketitle

\makeatletter
\renewcommand\@makefnmark%
{\mbox{\textsuperscript{\normalfont\@thefnmark)}}}
\makeatother

\section{Introduction}

In \cite[Theorem 3.5]{NAFA2-Kang}, we can see the following non-Archimedean 
fuzzy version of the classical Mazur-Ulam theorem:

\begin{theorem}\label{NAF}
Let $X$, $Y$ be non-Archimedean fuzzy anti-2-normed spaces over a certain type of non-Archimedean field $\K$. If both $X$ and $Y$ are strictly convex, then any centred fuzzy 2-isometry $f: X \to Y$ is an additive map.
\end{theorem} 

The class of strictly convex fuzzy anti-2-normed spaces was introduced in that paper (cf. \cite[Definition~2.5]{NAFA2-Kang}), although there appeared no examples there. In this note, we prove:

\begin{proposition}
There are no such strictly convex spaces at all.
\end{proposition}


As a consequence, the statement of the above Theorem is void. 

Let us also point out that the situation is similar in the different non-Archimedean versions of the Mazur-Ulam theorem that have appeared in recent years (cf. \cite{PRECISION} and references in \cite{NAFA2-Kang}).

\section{Voidness of the notion of strictly convex fuzzy space}

We reflect the following definition just for the sake of completeness.

\begin{definition}\label{defNAFA2} \rm 
A {\em non-Archimedean fuzzy anti-2-normed space} 
is a linear space $X$ over a non-Archimedean field $( \K, |\cdot |) $ together 
with a fuzzy anti-2-norm; that is to say, with a function $N:X^2\times\R\to[0,1]$ 
such that, for all $x,y\in X$ and all $s,t\in\R$,
\begin{enumerate}
\item[(A2N-1)] if $t\leq 0,$ then $N(x,y,t)=1$,
\item[(A2N-2)] if $t>0$, then $N(x,y,t)=0$ if and only if $x$ and $y$ are linearly dependent,
\item[(A2N-3)] $N(x,y,t)=N(y,x,t),$
\item[(A2N-4)] $N(x,cy,t)=N(x,y, t/|c|)$ for any non-zero $c \in \K$
\item[(A2N-5)] $N(x,y+z,\max\{s,t\})\leq \max\{N(x,y,s),N(x,z,t)\}$,
\item[(A2N-6)] $N(x,y,*)$ is a non-increasing function of $\R$ and $\lim_{t\to\infty}N(x,y,t)=0$. 
\end{enumerate}
\end{definition}

\begin{definition}\cite[Definition 2.5]{NAFA2-Kang} \rm A non-Archimedean fuzzy anti-2-normed space $(X,N)$ is {\em strictly convex} if
\begin{equation}\label{def022} 
N(x,y,s)=N(x,z,t) = N(x,y+z,\max\{s,t\}) \quad \Rightarrow \quad y=z \ \mbox{ and } \ s=t \ . 
\end{equation}
\end{definition}

\begin{proposition}\label{NoHaySC}
There are no strictly convex spaces at all --in the sense of the above Definition.
\end{proposition}

\begin{proof}
Any non-Archimedean fuzzy anti-2-norm $N$ satisfies, by (A2N-2), that, 
for any $x\in X$, and any $ s, t\in(0,\infty)$, 
$$
N(x,-x,s)= N(x,2x,t)=N(x,x,\max\{s,t\})=0 .
$$

As these equalities are valid {\em for any} $s, t\in(0,\infty)$, it is clear that 
no fuzzy anti-2-normed space $(X,N)$ may fulfil condition~(\ref{def022}) --not even the zero linear space. 
\end{proof}


\end{document}